\pdfoutput=1
\RequirePackage[l2tabu, orthodox]{nag}

\documentclass[reqno,tbtags]{amsart}

\usepackage{lmodern}
\usepackage[T1]{fontenc}
\usepackage[utf8]{inputenc}
\usepackage[english]{babel}
\usepackage{microtype} %Better font spacing

\usepackage{amsmath,amssymb,amsthm,mathrsfs,latexsym,mathtools,mathdots,xcolor,
	booktabs,array,enumerate,tikz,bm,extarrows,tkz-graph,hyperref}
\usepackage[centertableaux]{ytableau}

\usetikzlibrary{arrows,positioning}
\usepackage[capitalize,noabbrev]{cleveref}

\usepackage[colorinlistoftodos]{todonotes}

\newtheorem{theorem}{Theorem}
\newtheorem{proposition}[theorem]{Proposition}
\newtheorem{lemma}[theorem]{Lemma}
\newtheorem{corollary}[theorem]{Corollary}

\theoremstyle{definition}
\newtheorem{definition}[theorem]{Definition}
\newtheorem{example}[theorem]{Example}
\newtheorem{remark}[theorem]{Remark}
\newtheorem{question}[theorem]{Question}
\newtheorem{problem}[theorem]{Problem}

\newcommand{\defin}[1]{\textcolor{blue}{\emph{#1}}}
\newcommand{\mdefin}[1]{\textcolor{blue}{#1}}

\newcommand{\qbinom}{\genfrac{[}{]}{0pt}{}}

\newcommand{\setN}{\mathbb{N}}
\newcommand{\setZ}{\mathbb{Z}}

\newcommand{\schurS}{\mathrm{s}}

\newcommand{\modK}{{\tilde{K}}}

\newcommand{\avec}{\mathbf{a}}
\newcommand{\bvec}{\mathbf{b}}

\newcommand{\+}{{\scriptstyle{+}}}
\newcommand{\minus}{ {\scriptstyle{-}}}

\newcommand{\thsup}{\textnormal{th}}
\newcommand{\rskArrow}{\;\xrightarrow{\scriptstyle{\mathrm{RSK}}}\;}

\newcommand{\symS}{S}

\newcommand{\SYT}{\mathrm{SYT}}
\newcommand{\DES}{\mathrm{Des}}

\newcommand{\matrixSet}{\mathcal{M}}

\newcommand{\SSYT}{\mathrm{SSYT}}

\newcommand{\SHST}{\mathrm{SHST}}
\newcommand{\GT}{\mathrm{GT}}
\newcommand{\PPset}{\mathcal{PP}}
\newcommand{\SM}{\mathrm{SM}}

\DeclareMathOperator{\partitionN}{\kappa}

\DeclareMathOperator{\rw}{rw}
\DeclareMathOperator{\rev}{rev}
\DeclareMathOperator{\maj}{maj}
\DeclareMathOperator{\sign}{sgn}
\DeclareMathOperator{\inv}{inv}
\DeclareMathOperator{\rot}{rot}	%One-step rotation to the right
\DeclareMathOperator{\imaj}{imaj}
\DeclareMathOperator{\charge}{charge}
\DeclareMathOperator{\depth}{\delta}
\DeclareMathOperator{\cocharge}{cc}

\DeclareMathOperator{\length}{\ell}
\DeclareMathOperator{\PPbij}{\rho}
\DeclareMathOperator{\toggle}{\tau}

\definecolor{c1}{RGB}{127, 250, 185}
\definecolor{c2}{RGB}{233, 136, 247}
\definecolor{c3}{RGB}{247, 151, 99}
\definecolor{c4}{RGB}{86, 172, 252}
\definecolor{c5}{RGB}{169, 119, 230}
\definecolor{c6}{RGB}{85, 189, 79}

\title{Promotion and cyclic sieving on families of SSYT}

\author{Per Alexandersson}
\email{per.w.alexandersson@gmail.com}
\address{Department of Mathematics,
Stockholm University,
S-10691, Stockholm, Sweden}

\author{Ezgi Kantarci Oğuz}
\email{ezgikantarcioguz@gmail.com}
\address{Dept. of Mathematics, 
Royal Institute of Technology, 
SE-100 44 Stockholm, Sweden}

\author{Svante Linusson}
\email{linusson@math.kth.se}
\address{Dept. of Mathematics, 
Royal Institute of Technology, 
SE-100 44 Stockholm, Sweden}

\date{\today}

\begin{document}

\begin{abstract}
	We examine a few families of semistandard Young tableaux, for which 
	we observe the cyclic sieving phenomenon under promotion.
	
	The first family we consider consists of stretched hook shapes,
	where we use the cocharge generating polynomial as CSP-polynomial.
	
	The second family we consider consists of skew shapes, consisting of rectangles.
	Again, the charge generating polynomial together
	with promotion exhibits the cyclic sieving phenomenon.
	This generalizes earlier result by B.~Rhoades and later
	B.~Fontaine and J.~Kamnitzer.
	
	Finally, we consider certain skew ribbons,
	where promotion behaves in a predictable manner.
	This result is stated in form of a bicyclic sieving phenomenon.
	
	One of the tools we use is a novel method for computing charge of 
	skew semistandard tableaux, in the case when every number in the 
	tableau occur with the same frequency.
\end{abstract}

\maketitle

\setcounter{tocdepth}{2}
\tableofcontents

\section{Introduction}

The cyclic sieving phenomenon has been studied extensively since
its introduction in 2004 \cite{ReinerStantonWhite2004}.
Briefly, this phenomenon relates a cyclic group action on a set of combinatorial objects 
with the values at roots of unity, on some $q$-analog of the cardinality of the set.

\subsection{Cyclic sieving}

The notion of cyclic sieving is defined as follows.
\begin{definition}[Cyclic sieving, Reiner--Stanton--White \cite{ReinerStantonWhite2004}]
Let $X$ be a set of combinatorial objects and $C_n$ 
be a cyclic group of order $n$ acting on $X$. 
We say that the \defin{triple} $\mdefin{(X,C_n,f(q))}$ \defin{exhibits the cyclic sieving phenomenon, (CSP)}
if for all $d \in \setZ$,
\begin{align}\label{eq:cspDef}
|\{ x\in X : g^d \cdot x = x \}| = f(\xi^d)
\end{align}
where $\xi$ is a primitive $n^\thsup$ root of 
unity and $f(q)\in \setN[q]$.

\end{definition}
A few selected relevant CSP-triples are listed in \cref{tab:cspTable},
together with the new results in this article.

One particular instance of CSP is the set of semistandard Young tableaux of a fixed rectangular shape,
and a $q$-analog given by the $q$-hook-content formula. 
The cyclic group action is the promotion operator (denoted $\partial$, defined in \cref{sec:prelTab}),
see  B.~Rhoades~\cite{Rhoades2010}.
Rhoades' result was further refined in \cite{FontaineKamnitzer2013} where cyclic sieving on 
rectangular SSYT with a fixed content vector was considered.

In a recent preprint, \cite{AlexanderssonPfannererRubeyUhlin2020x},
it is proved that for a fixed skew shape $\lambda/\mu$, 
there exist \defin{some} cyclic group action $C_n$ of order $n$, such that 
\begin{equation}\label{eq:APRU}
 \left( \SYT(n\lambda/n\mu), C_n, f^{n\lambda/n\mu}(q) \right)
\end{equation}
is a CSP-triple, where $f^{n\lambda/n\mu}(q)$ is a skew Kostka--Foulkes polynomial.
Note that when $\mu=\emptyset$, $f^{n\lambda}(q)$ can be computed by the $q$-hook formula.
The notion of multiplying every part in a partition with a factor is commonly referred to
as \defin{stretching}, and has a certain geometric interpretation.
It is therefore natural to explore situations similar to \eqref{eq:APRU}.

\subsection{Main results}

There has been less research on non-rectangular shapes---promotion
usually have a very high order, which makes it unlikely to 
find nice instances of CSP.

One case which do behave nicely is the case of hook shapes.
Hook semistandard tableaux have been considered in
\cite{BennettMadillStokke2014,PresseyStokkeVisentin2016,OhPark2019}.
In this paper, we consider certain semistandard Young tableaux where 
the shape is a \defin{stretched hook}, i.e.,
of the form $\lambda = ((n+1)a, n^b)$ for some non-negative $a,b \geq 1$.

The first main theorem is as follows.
\begin{theorem}[\cref{cor:cspOnStretchedHooks} below]
 Let $\lambda = ((n+1)a, n^b)$, and let $\nu=n^{a+b+1}$. 
 Then 
  \[
  \left( \SSYT(\lambda, \nu), \langle \partial \rangle, q^{-n \binom{b}{2}} \modK_{\lambda,\nu}(q) \right)
 \]
exhibits the cyclic sieving phenomenon, where
\[
 \modK_{\lambda,\nu}(q) = \sum_{T \in  \SSYT(\lambda, \nu)} q^{\cocharge(T)}
 = q^{n \binom{b}{2}} \prod_{i=1}^a \prod_{j=1}^b \frac{[i+j+n-1]_q}{[i+j-1]_q}.
\]
\end{theorem}
Here $\modK_{\lambda,\nu}(q)$ is a \emph{modified Kostka--Foulkes polynomial},
given as a \defin{cocharge}-generating function.
In fact, by \cref{lem:KFrootOfUnityValues}, both the traditional 
and modified Kostka--Foulkes polynomial can be used.

We prove this result by using an equivariant map from 
stretched hook semistandard Young tableaux to a set of rectangular plane partitions,
where promotion correspond to a product of certain toggle operations.
We then use the fact that plane partitions together 
with such toggles exhibit cyclic sieving, see \cite{ShenWeng2018} (the case of order ideals of $2 \times [n]$ 
posets was considered earlier in \cite{RushShi2012}).
For some related open problems on plane partitions and cyclic sieving,
we refer to \cite{AbuzzahabKorsonChunMeyer2005,Hopkins2020}.

\begin{table}[!ht]
\centering
\begin{tabular}{p{4.2cm} l l l}
	\toprule
	\textbf{Set} & \textbf{Group} & \textbf{Polynomial} & \textbf{Ref.} \\ 
	\midrule 
	Binary words, $\mathrm{BW}(2n,n)$ &  $\rot$ & $\qbinom{2n}{n}_q$  &\cite{ReinerStantonWhite2004} \\
	Words, $W(kn,n^k)$ & $\rot$ & $\qbinom{kn}{n,n,\dotsc,n}_q$  &\cite{ReinerStantonWhite2004} \\
	\midrule 
	Rectangular, $\SSYT(a^b)$ &  $\partial$ & $q^{\ast}\schurS_{a^b}(1,\dotsc,q^{k-1})$  &\cite{Rhoades2010} \\
	Rectangular, $\SSYT(a^b,\gamma)$ &  $\partial^d$ &$q^\ast K_{a^b,\gamma}(q)$  &\cite{FontaineKamnitzer2013} \\
	Hooks $\SSYT((n-m,1^m),\gamma)$ &  $\partial^d$ &$\qbinom{nz(\gamma)-1}{m}_q$  &\cite{BennettMadillStokke2014} \\
	\midrule 
	Plane partitions & $\partial^\dagger$ 
	&$\prod_{\substack{1\leq i \leq a\\1 \leq j\leq b}} \frac{[i+j+n-1]_q}{[i+j-1]_q}$  &\cite{ShenWeng2018} \\
	\midrule 
	Matrices $M(n\nu,n^{m})$ & $\rot$ & 
	$\sum_{\lambda \vdash mn} K_{\lambda,n^m}(q) K_{\lambda,n\nu}(1)$ &\cite{Rhoades2010b} \\
	\midrule 
	\midrule 
	Stretched hooks & $\partial$ &  $\prod_{\substack{1\leq i \leq a\\1 \leq j\leq b}} \frac{[i+j+n-1]_q}{[i+j-1]_q}$ & \cref{cor:cspOnStretchedHooks} \\
	Disjoint rows   & $\partial$ &  $K_{n\lambda/n\mu,n^{|\lambda/\mu|}}(q)$ & \cref{thm:CSPonStretchedRibbons} \\
	Disjoint rectangles   & $\partial^d$ &  $K_{(a_1^{b_1}) \oplus \dotsb \oplus (a_r^{b_r}),\gamma}(q)$ & \cref{thm:rectanglesCSP} \\
	
	Certain two-row ribbon    & $\partial$ &  $\qbinom{n}{b}_t-[n]_t+[n-1]_q$ & \cref{cor:firstBiCSP} \\
	Certain three-row ribbon  & $\partial$ &  $[n-2]_t+(n-3)[n-1]_q$ & \cref{thm:secondBiCSP} \\
	\bottomrule
\end{tabular}
\caption{
	The group action is given by $k$-promotion, $\partial$, or an appropriate power of it,
	depending on the rotational symmetry of the content composition.
	The action $\partial^\dagger$ is defined via so-called ``toggles'', but can be mapped 
	in an equivariant manner to promotion on rectangular SSYT.
}\label{tab:cspTable}
\end{table}

We also show cyclic sieving under promotion on skew shapes,
consisting of a disjoint union of rows.
Let $\nu \vdash m$ and $n\geq 1$.
Let $\SM(\nu,n)$ be the set of skew semistandard Young tableaux where 
the shape $\lambda/\mu$ is a disjoint union of $\length(\nu)$
rows where row $j$ has length $n\nu_j$. 
Moreover, we ask that the content is given by $n^m$,
that is, there are $n$ boxes with label $j$, for each $j=1,\dotsc,m$.
\begin{theorem}[\cref{thm:CSPonStretchedRibbons} below]
We have that
\[
\left( \SM(\nu,n), \langle \partial \rangle, K_{\lambda/\mu,n^m}(q)  \right)
\]
is a CSP-triple, where $K_{\lambda/\mu,n^m}(q)$ is a Kostka--Foulkes polynomial.
\end{theorem}

In \cite{FontaineKamnitzer2013}, a cyclic sieving phenomenon involving 
$K_{a^b,\gamma}(q)$ is proved, where promotion act on the set $\SSYT(a^b,\gamma)$.
This result refines the earlier instance of CSP given in \cite{Rhoades2010}.
In \cref{thm:rectanglesCSP}, we give a generalization of the result by Fontaine--Kamnitzer to the case when promotion act on
skew shapes consisting of a disjoint union of rectangles.
As expected, Kostka--Foulkes polynomial indexed by a skew shape and a
composition is the main part of the CSP-polynomial.

Finally, in \cref{sec:biCSP}, we study promotion on two families of ribbon shapes,
where the order of promotion behaves nicely. We prove two instances of cyclic sieving.
In the last section, we give a few examples indicating that 
promotion on general ribbon shapes is most likely hard to analyze.

\subsection*{Acknowledgements}

The authors are grateful for the Institut Mittag--Leffler program
on \emph{Algebraic and Enumerative combinatorics}, spring 2020, funded by
\emph{Swedish Research Council (Vetenskapsrådet)}, grant 2016-06596.

The first author has been funded by \emph{Knut and Alice Wallenberg Foundation} (2013.03.07),
and by the \emph{Swedish Research Council}, grant 2015-05308.
The second and third author has been funded by \emph{Swedish Research Council}, grants 621-2014-4780 and 2018-05218.
\section{Preliminaries}

\subsection{Semistandard tableaux and Jeu-de-taquin promotion}\label{sec:prelTab}

With every partition $\lambda=(\lambda_1,\lambda_2,\dotsc,\lambda_k)$, we associate a \defin{Young diagram},
an array with $\lambda_i$ boxes on row $i$. Given two partitions $\lambda$ and $\mu$ with $\mu \subset \lambda$,
the \defin{skew-diagram} $\lambda / \mu$ is given by the cells in $\lambda$ that are not in $\mu$.
A \defin{semistandard Young tableau} of shape $\lambda$ (or $\lambda / \mu$)
is a filling of its boxes with positive integers such that each row is weakly
increasing from left to right and each column is strictly increasing from top to
bottom. We denote the set of semistandard tableaux of shape $\lambda$ by $\mdefin{\SSYT(\lambda)}$. 
A semistandard tableau that contains each of the numbers from $1$ to $n$ exactly once is called
\defin{standard}. 
The \defin{content} of a tableau $T$ is the composition $(\nu_1,\nu_2,\dotsc,\nu_n)$
where $\nu_i$ is given by the number of
occurrences of $i$ in $T$
We denote the set of all semistandard Young
tableaux of shape $\lambda$, content $\nu$ 
by $\mdefin{\SSYT(\lambda,\nu)}$.

With any tableau $T$, we associate a \defin{reading word} denoted $\rw(T)$,
a listing of its entries row by row, left to right, bottom to top.
Note that if $T$ is standard, its reading word is a permutation of $n$.
\begin{example}
A semistandard tableau of shape $(5,2,2)$ and its reading word.  
\[
T=
\ytableausetup{boxsize=0.9em}
\ytableaushort{
11234,
23,
34}, \qquad \rw(T)=34\,23\,11234.
\]
\end{example}

We act on semistandard Young tableaux by \defin{jeu-de-taquin promotion} (also called $K$-promotion)
was introduced in \cite{Schutzenberger1963,Schutzenberger1972}. 
We shall follow the definition given in \cite{BennettMadillStokke2014}.
Our definition differs from theirs in the sense that we start with removing all 1s---
the inverse of the operation in their work. 
Throughout this paper, we will refer to this simply as promotion, denoted by $\partial$.

Jeu-de-taquin promotion works as follows. First all entries on $T$ labeled $1$ are
replaced by dots. Then the dots are moved to the outside corners by repeatedly changing
places with the smaller of the entries to the right or below. 
If both are equal, exchange is made with the entry below to maintain tableau rules.
Once no more exchanges are possible, all numbers are decremented by one,
and dots are replaced by $n$, as seen in the following example.
\begin{example} Jeu-de-taquin promotion
	\[ \ytableausetup{boxsize=.9em}
	\begin{ytableau}
	{\color{red}{1}}& {\color{red}{1}}  & 2&3&4\\
	2&3\\
	3&4
	\end{ytableau} \quad \rightarrow \quad
	\begin{ytableau}
	\cdot & \cdot  & 2&3&4\\
	2&3\\
	3&4
	\end{ytableau} \quad \rightarrow \quad
	\begin{ytableau}
	\cdot & 2  & \cdot&3&4\\
	2&3\\
	3&4
	\end{ytableau} \quad \rightarrow \quad
	\begin{ytableau}
	2& 2  & 3&\cdot &4\\
	\cdot &3\\
	3&4
	\end{ytableau}	\]
	
	\[  \rightarrow \quad
	\begin{ytableau}
	2& 2  & 3& 4&\cdot\\
	3 &3\\
	\cdot &4
	\end{ytableau}  \quad \rightarrow \quad
	\begin{ytableau}
	2& 2  & 3& 4&\cdot\\
	3 &3\\
	4 &\cdot
	\end{ytableau}
	\quad \rightarrow \quad
	\begin{ytableau}
	1& 1  & 2& 3& {\color{red}{4}}\\
	2 &2\\
	3 & {\color{red}{4}}
	\end{ytableau}\]
\end{example}

\subsection{Charge, cocharge and Kostka--Foulkes polynomials}

\begin{definition}
For a word $\pi$ define \defin{major index} as $\mdefin{\maj(\pi)}\coloneqq \sum_{d\in \DES(\pi)} d$. 
If $\pi$ is also a permutation of length $n$, define \defin{charge} 
as $\mdefin{\charge(\pi)}\coloneqq \maj(\rev(\pi^{-1})$ and \defin{cocharge} 
as $\mdefin{\cocharge(\pi)}\coloneqq  \binom{n}{2}-\charge(\pi)$. 
For a word $w$ with content given by a partition $(v_1,\dotsc,v_n)$
we can define $v_1$ \defin{standard subwords} as follows. 
The first standard subword is obtained by starting with the rightmost 1 and then go 
left to find a 2 then a 3 etc. 
If there is no $i+1$ to the left of $i$ in the word then we wrap and 
take the rightmost $i+1$ instead and continue for as long as possible.
These letters, in their original order, makes up the first standard subword $w_1$. 
The standard subword $w_r$ is formed in the same way but not using any 
letters that already belong to any of the standard 
subwords $w_1,\dotsc,w_{r-1}$. 
We then extend the definition of charge and cocharge to the word $w$ 
by $\mdefin{\charge(w)} \coloneqq \sum_{i} \charge(w_i)$, 
$\mdefin{\cocharge(w)}\coloneqq\sum_i\cocharge(w_i)$. 
In this work we focus on the case where $w$ is the reading word of a tableaux.
\end{definition}

\begin{example}
	The tableau $T \in \SSYT(6,3,3)$ below has the standard subwords 
	$w_1=35214$, $w_2=42135$, $w_3=31245$ 
	and $\charge(\rw(T))=2+4+7=13$, $\cocharge(\rw(T))=8+6+3=17$.
	\begin{equation}
	\ytableausetup{boxsize=0.9em}
	\ytableaushort{
		111234455,
		223,
		345}
		\qquad 	\rw(T)=345223111234455 
	\end{equation}
\end{example}

We shall also use the following alternative method of computing cocharge.
The cocharge of a permutation $\pi \in \symS_n$ can be computed as follows.
We now define the \defin{cocharge value} of $j$, denoted $cc(\pi,j)$.
\begin{equation}\label{eq:cochargeValues}
cc(\pi,j) \coloneqq 
\begin{cases}
0 &\text{ if } j=1 \\
cc(\pi,j-1) &\text{ if $j-1$ appear to the left of $j$ in $\pi$} \\
cc(\pi,j-1)+1 &\text{ otherwise}.
\end{cases}
\end{equation}
This uniquely defines the values of $cc(\pi,1)$, $cc(\pi,2),\dotsc,cc(\pi,n)$.
Finally, $\cocharge(\pi)$ is the sum of these cocharge values.

\begin{example}[Computing cocharge]
Let $\pi = ({4, 8, 6, 9, 7, 2, 3, 1, 5})$.
The cocharge value $cc(\pi,j)$ is written as the subscript of $j$:
\[
4_2, 8_4, 6_3, 9_4, 7_3, 2_1, 3_1, 1_0, 5_2.
\]
The total sum of the subscripts is $20$, 
so the cocharge of $\pi$ is $20$. 
\end{example}

\begin{lemma}\label{lem:cochargeRotation}
Let $\pi \in \symS_n$, and let $\rot$ act by cyclic shift to the right.
Then 
\[
\cocharge(\pi) =
\begin{cases} 
\cocharge(\rot(\pi))-1 &\text{ if $\pi(n) \neq 1$} \\
\cocharge(\rot(\pi))+n-1 &\text{ if $\pi(n) = 1$}.
\end{cases}
\]
In particular
$
\cocharge(\pi) \equiv \cocharge(\rot(\pi))-1 \mod n.
$
\end{lemma}
\begin{proof}
	This follows immediately from the recursion in \eqref{eq:cochargeValues}.
\end{proof}

The charge and cocharge statistics on tableux can be used to calculate a 
special family of polynomials called the \defin{Kostka--Foulkes polynomials}
and \defin{modified Kostka--Foulkes polynomials}, respectively.
\begin{align*}
	\mdefin{K_{\lambda/\mu,\nu}(q)} & \coloneqq \sum_{T \in \SSYT(\lambda/\mu,\nu)}q^{\charge(\rw(T))}\\
	\mdefin{\modK_{\lambda/\mu,\nu}(q)}&\coloneqq \sum_{T \in \SSYT(\lambda/\mu,\nu)}q^{\cocharge(\rw(T))}\\
\end{align*}
The modified and classical Kostka--Foulkes polynomials are related via the relation 
\begin{equation}\label{eq:modVsClassicalKF}
\modK_{\lambda/\mu, \nu}(q) = q^{\partitionN(\nu)} K_{\lambda/\mu,\nu}(q^{-1}),
\end{equation}
where 
\[
\mdefin{\partitionN(\lambda)} \coloneqq \sum_{j} \binom{\lambda_j'}{2},
\]
for the conjugate partition $\lambda'$.
Since $\lambda_j'$ is the number of boxes in the $j$th column of $\lambda$, we get $\partitionN(n\lambda) = n \partitionN(\lambda)$.
Moreover, for $\lambda = (a+1,1^b)$, then $\partitionN(\lambda)=b(b+1)/2$.

Finally, permuting the entries of the content partition does not change the 
Kostka--Foulkes polynomial. That is,
\[
 \mdefin{K_{\lambda/\mu,\sigma(\nu)}} \coloneqq K_{\lambda/\mu,\nu}
\]
for any permutation $\sigma$. Warning! Note that we may only compute 
Kostka--Foulkes polynomials via charge on SSYT if the content is a partition.

\begin{lemma}\label{lem:KFrootOfUnityValues}
Let $\xi$ be an $n^\thsup$ root of unity, such that $K_{\lambda/\mu,n\nu}(\xi)$ is a real number.
Then $\modK_{\lambda/\mu,n\nu}(\xi)  = K_{\lambda/\mu,n\nu}(\xi)$.
\end{lemma}
\begin{proof}
By using \eqref{eq:modVsClassicalKF}, we have that
\begin{equation}
\modK_{\lambda/\mu,n\nu}(\xi) = 
\xi^{n\cdot \partitionN(\nu)} K_{\lambda/\mu,n\nu}(\xi^{-1}) = 
K_{\lambda/\mu,n\nu}(\xi).
\end{equation}
In the last equality, we use the fact that $\xi^{-1} = \overline{\xi}$,
and that $K_{\lambda/\mu,n\nu}(\xi)$ is real.
\end{proof}

\section{Promotion on stretched hooks}

The main results of this section are a new way of computing the charge of a word with 
rectangular content and bijection between stretched hooks and plane partitions
which allows to prove cyclic sieving for stretched hooks. 

Let $\SSYT(\lambda, \mu)$ denote the set of 
semistandard Young tableaux of shape $\lambda$ and content $\mu$.
The \defin{Kostka coefficient}, $\mdefin{K_{\lambda\mu}}=K_{\lambda\mu}(1)$
is therefore given by $|\SSYT(\lambda, \mu)|$.

Given $a,b\geq 0$ and $n\geq 1$, let $\SHST(a,b,n)$ 
(\defin{stretched hook semistandard tableaux}) be the set 
\[
\mdefin{\SHST(a,b,n)} \coloneqq \{T \in \SSYT((na+n)n^b, n^{a+b+1}) \}.
\]
We let $k$-promotion $\partial$ act on $\SHST(a,b,n)$.
\begin{example} \label{stretchedpromotion}
The six tableaux in $\SHST(1,2,2)$ form two $3$-cycles under promotion.
\[	\ytableausetup{boxsize=0.9em}
\ytableaushort{1122,33,44}\quad \xrightarrow{\partial} \quad
\ytableaushort{1144,22,33}\quad \xrightarrow{\partial}\quad
\ytableaushort{1133,22,44}\quad \xrightarrow{\partial}\quad
{\color{red!50}\ytableaushort{1122,33,44}}
\]
\[\ytableaushort{1123,23,44}\quad \xrightarrow{\partial}\quad
\ytableaushort{1124,23,34}\quad \xrightarrow{\partial}\quad
\ytableaushort{1134,22,34}\quad \xrightarrow{\partial}\quad
{\color{red!50}\ytableaushort{1123,23,44}}
\]
\end{example}

Note that in the example above, repeating promotion $a+b$ times lets us 
recover our original tableaux.  Now we will show that this is always the case. 
\begin{proposition}[Order of promotion]
For any $T \in \SHST(a,b,n)$, we have $\partial^{a+b}(T)=T$. 
\end{proposition}
\begin{proof}
The $n$ 1s are all in the first positions of the first row. 
When they are removed 2s will fill their positions. Thus the promotion operation
on stretched hook shapes is independent of the entries of the first row, as the
total content is fixed. Promotion acts on the remaining $bxn$ shape by doing
promotion with alphabet $2,3,\ldots,a+b+1$. This is equivalent to doing promotion
on rectangular semi standard Young tableaux filled with alphabet $[a+b]$
which has the property $\partial^{a+b}(T)=T$.
\end{proof}

Deleting the first row and decrementing entries by one gives a bijection 
between $\SHST(a,b,n)$ and $\SSYT(n^b,a+b)$ that commutes with promotion, 
which we used in proving our proposition above. 
What makes the stretched hook shapes still interesting is the 
use of cocharge statistic and its relation with the promotion operation. 

Let 
$\mdefin{|\SHST(a,b,n)|_q} \coloneqq \sum_{T \in \SHST(a,b,n)} q^{\cocharge(T)}$.
Note that this is a (modified) Kostka--Foulkes polynomial.

\begin{example} 
The six stretched hook semistandard tableaux given in 
\cref{stretchedpromotion} above have 
cocharges $6$, $10$, $8$, $7$, $8$, and $9$, respectively. 
The corresponding polynomial is 
\[
|\SHST(1,2,2)|_q =q^6+q^7+2q^8+q^9+q^10=q^6(1+q+q^2)(1+q^2).
\]
\end{example}

Given a word $w$ on $[k]$ and an integer $j\geq 1$,
consider the subword of $w$ consisting of entries in $\{j,j+1\}$.
By repeatedly removing consecutive pairs $(j+1,j)$,
we end up with a word of the form $j^r\; (j+1)^s$.
We let $\depth_j(w)$ be the value of $s$, and set
\begin{equation}\label{eq:descentSequence}
\mdefin{\depth(w)}  \coloneqq (\depth_1(w),\depth_2(w),\ldots,\depth_j(w))
\end{equation}
as the \defin{depth sequence} of $w$. 
The depth sequence gives a new way of computing
the charge of any word with rectangular content.

\begin{theorem}[Charge for rectangular content]\label{T:Rectang_charge}
Suppose $w$ is a word with content $k^n$. Then 
\begin{equation}\label{eq:chargeFromPhi}
\charge(w) = \sum_{j} \depth_j(w)(k-j).
\end{equation}
\end{theorem}
\begin{proof} 
To calculate charge, we first divide $w$ into standard subwords, 
which in this case all have the same length $n$, then calculate the 
charge at each subword. 
The contribution of a letter $j$ to charge just depends on 
whether it comes before or after $j+1$ in its subword. 
So, limiting our attention to $j$s we can see the subword selection process as an 
ordering $\sigma$ on the $n$ entries labeled $j$, each of which then claims the first 
unclaimed $j+1$ to the left, looping around if necessary.
We shall now show that the total contribution of entries $j$ to the charge is 
independent of $\sigma$. It follows that we can use \eqref{eq:chargeFromPhi},
as the pairings with $j+1$ coming before $j$ do not contribute to the charge.

Consider an ordering $\sigma$, and exchange the order of two consecutively ordered entries. 
They either might be paired up with their original $(j+1)$s, in which case total 
charge does not change, or the largest entry in each pair are interchanged. 
The latter can happen if the relevant entries have one of the following four relative orderings:
\[
(j+1) \, (j+1)\, j\, j, \qquad
(j+1)\, j\, j\, (j+1), \qquad
j\, (j+1)\, (j+1)\, j, \qquad
j\, j \,(j+1)\, (j+1).
\]
It is straightforward to verify that the order exchange 
does not change the total contribution to charge.  
\end{proof}

\begin{example} 
Note that the first row of tableaux from \cref{stretchedpromotion} 
have depth sequences $(2,0,0)$, $(0,0,2)$, $(0,2,0)$,
and their charges are  $6,2$ and $4$ as expected.
For this example promotion rotates the depth sequence. 
We will next look at an easy way to calculate the depth sequence 
using only the first row, which will show that promotion always rotates it.
\end{example}

\begin{theorem} \label{T:depth}
Let $T \in \SHST(a,b,n)$. 
Then the depth sequence of the reading word of $T$ only depends on the first row of $T$. 
In particular, $\depth_j$ is given by the number of $j+1$s on the first row.
\end{theorem}
\begin{proof}
When we consider the entries $j$ and $j+1$ in the reading word,
the $j+1$s on the first row will not be matched to any $j$s,
as they are at the end of the word, $\depth_j$ is at least that. 
We claim that any other $j+1$ will be paired to a $j$
under the algorithm, and therefore will not contribute to charge. 

Outside the first row, we have $b$ rows of length $n$ each.
Consider the $n$ columns. In each column, there can be a $j$,
and $j+1$, both or neither. If we have a $j+1$ and a $j$, $j$ will appear directly 
above $j+1$, therefore will be after $j+1$ in the reading word.
Also, if we have $l$ columns with just $j+1$,
we can have at most $n-l$ cells labeled $j$ in the first $n$ columns, so there are at
least $l$ corresponding $j$s on the first row, meaning all single $j+1$s will be matched as well.
\end{proof}

\begin{corollary} \label{C:charge_SHST}
For $T \in \SHST(a,b,n)$ we have that 
\[ 
\charge(T)=((a+b+2)a+1)n-sum_1(T),
\]
where $sum_1(T)$ is the sum of all entries of the first row.
\end{corollary}
\begin{proof}
By \cref{T:Rectang_charge} and \ref{T:depth} 
$\charge(T) + sum_1(T)= 
\sum_{j\ge 2} \depth_{j-1}(\rw(T))(a+b+1-(j-1)) +n+ \sum_{j\ge 2}\depth_{j-1}(\rw(T))j
= \sum_{j\ge 2} \depth_{j-1}(\rw(T))(a+b+2) +n=an(a+b+2)+n$,
where the last identity uses \cref{T:depth} once more. 
\end{proof}

\begin{corollary} 
For $T \in \SHST(a,b,n)$ we have that 
\[
\depth \circ \partial \circ T  = \rot^{-1} \circ \depth \circ T.
\]
That is, promotion rotates the depth sequence one unit to the left.
\end{corollary}
\begin{proof} 
This follows as the promotion on the first row just acts by removing 
any $2$s, subtracting $1$ from larger entries, and adding $a+b+1$s to replace the $2$s.
\end{proof}
Note that the result in this corollary is reminiscent of 
the notion of cyclic descent sets.

\subsection{Plane partitions, Gelfand--Tsetlin polytopes, and cyclic sieving}

A \defin{plane partition} $\pi$ is a rectangular array with non-negative integer entries, 
such that rows and columns are weakly increasing.
Let $\mdefin{\PPset(a,b,n)}$ denote the set of plane partitions
within a $a\times b$-rectangle, with maximal entry at most $n$.
A classical result due to MacMahon \cite[p. 659]{MacMahon1896} states that
\begin{equation}\label{eq:macMahon}
\mdefin{M_q(a,b,n)} \coloneqq \sum_{\pi \in \PPset(a,b,n)} q^{|\pi|} 
= \prod_{i=1}^a \prod_{j=1}^b \frac{[i+j+n-1]_q}{[i+j-1]_q}.
\end{equation}
Here, $\mdefin{|\pi|}$ denotes the total sum of entries in the plane partition.

A result due to R.~Stanley~\cite{StanleyEC2} implies that for $\lambda = b^a$,
MacMahons generating function is essentially a principal 
specialization of a rectangular Schur polynomial:
\begin{equation}
M_q(a,b,n) = q^{-\partitionN(\lambda)} \schurS_\lambda(1,q,q^2,\dotsc,q^{a+b-1}).
\end{equation}
\begin{definition}
We now define a bijection $\PPbij$ between $\SHST(a,b,n)$ and $\PPset(a,b,n)$.
Given a stretched hook $T$, we first construct a corresponding GT-pattern, 
which we denote by $\GT(T)$. The $i^\thsup$ element (from the left) 
in row $j$ (from below) in $\GT(T)$ is the number of boxes in row $i$ of $T$
that are labeled $j$ or smaller.
The GT-pattern will thus have $a+b+1$ rows, the top row starts with $n(a+1)$,
followed by $b$ entries labeled $n$ and ending with $a$ zeros. 
Below the zeros at the end of the first row there will by definition
be a triangle of zeros and below the $n$s there will be a triangle of $n$s. 
Furthermore, the leftmost entry on each row can be determined
by the rest of the entries since the sum in row $j$ must be $jn$. 
Eliminating these entries leaves a sideways plane partition of
size $a\times b$ with entries less than or equal to $n$,
which we denote by $\PPbij(T)$.
The inverse of this map is equally easy to define.
\end{definition}

\begin{example}
A tableau $T \in \SHST(2,3,4)$, the corresponding
GT-pattern $\GT(T)$ and the corresponding plane partition $\PPbij(T)$ (shown in bold).
\begin{equation}
\ytableausetup{boxsize=0.9em}
\ytableaushort{
	111122334566,
	2234,
	3455,
	4566}
\quad 
\begin{matrix}
12 && 4 && 4 && 4 && 0&& 0\\
& 10 && 4 && 4 && \mathbf{2}&&0\\
&&   9  && 4 && \mathbf{2} && \mathbf{1}\\
&&&     8 && \mathbf{3} &&\mathbf{1}\\
&&&&      6 && \mathbf{2}\\
&&&&&      4
\end{matrix}
\end{equation}
\end{example}
Now let us compare the charge statistic in stretched hook tableaux
with the sum statistic on the corresponding plane partition. Consider the two extremal cases where the
pattern consists only of $0$s or $n$s.
\begin{equation}
\ytableausetup{boxsize=0.9em}
\ytableaushort{
	111122223333,
	4444,
	5555,
	6666}
\quad 
\begin{matrix}
12 && 4 && 4 && 4 && 0&& 0\\
& 12 && 4 && 4 &&\mathbf{0}&& 0\\
&&   12  && 4 && \mathbf{0} && \mathbf{0}\\
&&&     12 && \mathbf{0} &&\mathbf{0}\\
&&&&       8 && \mathbf{0}\\
&&&&&       4
\end{matrix}
\end{equation}
By \cref{C:charge_SHST} this example gives us the maximum charge 
value $36$ on $\SHST(2,3,4)$, which is when the depth sequence is $(4,4,0,0,0)$. 
The corresponding plane partition $\PPbij(T)$, is formed of all zeros, and has sum $|\PPbij(T)|=0$. 
In general, the reading word with with maximal charge has the minimal
depth sequence formed of $a$ $n$s followed by zeros,
and minimal sum of the entries in the first row,
giving us $\charge(T)=a(a+2b+1)n/2$.
\begin{equation}
\ytableausetup{boxsize=0.9em}
\ytableaushort{
	111155556666,
	2222,
	3333,
	4444}
\quad 
\begin{matrix}
12 && 4 && 4 && 4 && 0&& 0\\
& 8 && 4 && 4 &&\mathbf{4}&& 0\\
&&   4  && 4 && \mathbf{4} && \mathbf{4}\\
&&&     4 && \mathbf{4} &&\mathbf{4}\\
&&&&       4 && \mathbf{4}\\
&&&&&       4
\end{matrix}
\end{equation}
This example gives us the minimum charge value $12$ on $\SHST(2,3,4)$
which is when the depth sequence is $(0,0,0,4,4)$ and the sum of the first row is maximal.
The corresponding plane partition $\PPbij(T)$, is
formed of all $n$s, and has sum $|\PPbij(T)|=16$.
In general, the reading word has depth sequence formed
of $b$ zeroes followed by $a$ $n$s, which by \cref{T:Rectang_charge} 
gives $\charge(T)=a(a+1)n/2$, and the corresponding plane
partition has sum  $|\PPbij(T)|=abn$, giving us the total $a(a+2b+1)n/2$

Note that in both cases the sum of the charge of the tableaux
and the sum of the entries on the corresponding rectangle is the same. 
We will show that this is always the case.

\begin{theorem} 
The bijection described above, sends $T\in\SHST(a,b,n)$
to $\PPbij(T)\in \PPset(a,b,n)$ such that 
\begin{align}
\charge(T)+|\PPbij(T)|  &=\frac{n(a+2b+1)(a)}{2}\label{eq:chargeAsPPSum}\\
\cocharge(T)-|\PPbij(T)|&=\frac{n(b+1)(b)}{2}\label{eq:chargeAsPPSum2}
\end{align}
\end{theorem}
\begin{proof} 
We will focus on the first identity, the second
follows as $\charge(T)+\cocharge(T)=\frac{n(a+b)(a+b+1)}{2}$. 

We already showed that the identity holds if $\PPbij(T)$ is formed by $n$s only.
Assume \eqref{eq:chargeAsPPSum} holds when the sum of the
entries is larger than or equal to $M$. 
Consider a pattern with $|\PPbij(T)|=M-1$. 
As this is not the maximal case, there is an entry that we can
increase by $1$ without violating the rules. 
This operation corresponds to replacing the rightmost $k$ by $k+1$ for some $k$ on the first row, 
and replacing a $k+1$ by $k$ on a row below, where the number $k$ and the place of 
the above entry depend on the choice of the coordinate we increase. 
This increases the sum of the first row by $1$, 
so by \cref{C:charge_SHST} the sum of $\charge(T)$ and $|\PPbij(T)|$ stays the same.
\end{proof}

An immediate result that follows from this formula is the affect of
exchanging $a$ and $b$ on the cocharge polynomial.
\begin{corollary}[Conjugation Symmetry] 
Rotating the rectangle $\PPbij(T)$ gives us a bijection
between $\SHST(a,b,n)$ to $\SHST(b,a,n)$ preserving $|\PPbij(T)|$
and thus the cocharge. In particular, we have 
\[
|\SHST(a,b,n)|_q = q^{\frac{n(a^2-a-b^2+b)}{2}} |\SHST(b,a,n)|_q.
\]
\end{corollary}

Also, combining the above result with \cref{eq:macMahon}
gives us a way to calculate $|\SHST(a,b,n)|_q $ directly.
\begin{corollary}[Kostka--Foulkes polynomials as plane partitions]\label{cor:KFasPP}
We have the identity
\[
|\SHST(a,b,n)|_q = 
q^{n \binom{b}{2}} \prod_{i=1}^a \prod_{j=1}^b \frac{[i+j+n-1]_q}{[i+j-1]_q}=q^{n \binom{b}{2}}M_q(a,b,n).
\]
\end{corollary}

The main result in \cite{ShenWeng2018} is the following
instance of the cyclic sieving phenomenon on plane partitions, under an operation called toggles.

\begin{theorem}[See \cite{ShenWeng2018}]\label{thm:toggleCSP}
Let $\toggle$ act on $\PPset(a,b,n)$. Then
\begin{equation}
 \left(
 \PPset(a,b,n), \langle \toggle \rangle,  M_q(a,b,n)
 \right)
\end{equation}
is a CSP-triple.
\end{theorem}

Combining this with Corollary \ref{cor:KFasPP} we get our first main result.

\begin{corollary}[Cyclic sieving on stretched hook tableaux]\label{cor:cspOnStretchedHooks}
We have that $\partial$ act on $\SHST(a,b,n)$, with order $a+b$,
and the triple
\[
\left( \SHST(a,b,n), \langle \partial \rangle, q^{-n \binom{b}{2}} \SHST_q(a,b,n) \right)
\]
is a CSP-triple.
\end{corollary}
\begin{proof}
Note that the promotion action on any $T \in \SHST(a,b,n)$ is
determined by promotion on the rectangular shape obtained by
deleting the first row of $T$. Our bijection with plane partitions
matches the one described in \cite[Appendix 1]{Hopkins2019} where it
is shown that the toggle operation is mapped to $\partial^{-1}$.
The result follows as $q^{-n \binom{b}{2}} \SHST_q(a,b,n)=M_q(a,b,n)$.
\end{proof}
Note that the $n=1$ correspond to a special case of \cite{BennettMadillStokke2014}.

\subsection{Discussion and background}

The earliest reference connecting toggles and promotion is 
the article by A.~Kirillov and A.~Berenstein~\cite{KirillovBerenstein1996}.

Later in \cite{ShenWeng2018}, using completely different methods,
the authors prove that plane partitions in an $a\times b$-box with max size $n$
exhibit CSP under \emph{piecewise linear toggles}.
However, they do not mention the connection with promotion.

The connection between promotion, toggles and cyclic
sieving is made explicit in S.~Hopkins~\cite{Hopkins2019}, 
where he studies plane partitions with additional symmetry.
He also discusses the connection with \emph{promotion} 
and \emph{rowmotion} on posets considered in \cite{StrikerWilliams2012}.

\begin{question}
There is a promotion-type action on type $B$ minuscule poset ideals, see \cite{RushShi2012,Hopkins2020}.
This is the same as a type of toggle on symmetric plane partitions, and there are $2^n$
such plane partitions. Can this action also be realized as an action on SSYT?
\end{question}

\section{Promotion and cyclic sieving on skew shapes}

Let $\nu \vdash m$ be an integer partition, and let $\mdefin{\SM(\nu,n)}$
be the set of skew semistandard Young tableaux where 
the shape is a disjoint union of $\length(\nu)$
rows where row $j$ has length $n\nu_j$. The content is given by $n^m$,
so that $\SM(\nu,1)$ consists of skew standard Young tableaux.
Given a tableau $T \in \SM(\nu,n)$, we associate a matrix $M=M(T)$ via
\[
\mdefin{M_{ij}} = \text{number of entries in row $i$ of $T$ with value $j$}.
\]
By construction, 
\[
\sum_{j\geq 1} M_{ij} = n\nu_i \quad \text{ and } \quad \sum_{i\geq 1} M_{ij} = n,
\]
and we do in fact have a bijection between $\SM(\nu,n)$ and $\mdefin{\matrixSet(n\nu,n^{m})}$,
the set of non-negative integer matrices with row sums given by $n\nu$ and each column summing to $n$.
Given a matrix $M \in \matrixSet(n\nu,n^{m})$, we associate a biword $W$ where $\binom{i}{j}$
appears $M_{n+1-i,j}$ times and the entries in the biword are sorted lexicographically.
We then obtain a bijection between $\SM(\nu,n)$ and 
biwords of length $n|\nu|$ such that the content of the top 
row is $(n\nu_{\ell},\dotsc,n\nu_2,n\nu_1)$
and the content of the bottom row is $(n,n,\dotsc,n)$.
Moreover, if $T$ is mapped to $W=W(T)$ (via $M(T)$) then the 
reading word of $T$ is the bottom row of $W$.

\begin{example}
Here is a $T\in \SM(211,3)$, the corresponding matrix $M(T)$, and the biword $W(T)$.
\[
\ytableaushort{{\none}{\none}{\none}{\none}{\none}{\none}112334,{\none}{\none}{\none}134,224}
\qquad
\begin{pmatrix}
2 & 1 & 2 & 1 \\
1 & 0 & 1 & 1 \\
0 & 2 & 0 & 1 
\end{pmatrix}
\]
\[
\setcounter{MaxMatrixCols}{20}
\begin{pmatrix}
1&1&1 &2&2&2 &3&3&3&3&3&3 \\
2&2&4 &1&3&4 &1&1&2&3&3&4 
\end{pmatrix}
\]
\end{example}

\begin{proposition}\label{prop:rskOnSkewStuffAndCharge}
Let $\nu \vdash m$ and $n\geq 1$. 
The Robinson--Schensted--Knuth correspondence (RSK) gives a bijection
\begin{align*}
\SM(\nu,n) &\rskArrow 
\bigcup_{\lambda \vdash nm } \SSYT(\lambda,n^m) \times \SSYT(\lambda,n\overline{\nu}) \\
T &\rskArrow (P,Q).
\end{align*} 
where the insertion algorithm is performed on the biword associated with $T$,
and $\overline{\nu}$ denotes the reverse of $\nu$.
Furthermore, $\charge(T)=\charge(P)$.
\end{proposition}

\begin{proof}
Let $w$ be the bottom row of the biword associated with $T$, i.e., $w = \rw(T)$.
By definition, $\charge(T) = \charge(w)$.
Moreover, $w$ insert to the semistandard tableau $P$ under RSK,
so $w$ and $\rw(P)$ are Knuth-equivalent.
It is a property of charge that if $w_1$ and $w_2$ are 
Knuth-equivalent, then $\charge(w_1)=\charge(w_2)$, see e.g. \cite{Butler1994}.
This last property implies that $\charge(T)=\charge(P)$.
\end{proof}

We let the $q$-analogue of $|\SM(\nu,n)|$ be defined as
\begin{equation}\label{eq:qSM}
\mdefin{|\SM(\nu,n)|_q} \coloneqq \sum_{T \in \SM(\nu,n)} q^{\charge(T)}.
\end{equation}
By definition, $|\SM(\nu,n)|_q$ is the skew Kostka--Foulkes coefficient $K_{\lambda/\mu,n^m}(q)$,
where $\lambda/\mu$ is the skew shape defined by the disjoint 
rows of lengths given by the parts of $n\nu$.

\begin{corollary}\label{cor:skewCharge}
	Let $\nu \vdash m$ and $n\geq 1$. Then 
	\begin{equation}
	|\SM(\nu,n)|_q = \sum_{\lambda \vdash mn} K_{\lambda,n^m}(q) K_{\lambda,n\nu}(1)
	\end{equation}
\end{corollary}
\begin{proof}
The identity follows from \cref{prop:rskOnSkewStuffAndCharge},
and by using the fact that $K_{\lambda\nu}=K_{\lambda\overline{\nu}}$.
\end{proof}

Note that $K_{\lambda,\nu}(1)=K_{\lambda,\nu}$, the Kostka coefficients defined previously.
\begin{theorem}[Cyclic sieving on stretched skew shapes with disjoint rows]\label{thm:CSPonStretchedRibbons}
Let $\nu \vdash m$ and $n\geq 1$. Set 
$\lambda/\mu$ to be the skew shape defined by the disjoint 
rows of lengths given by the parts of $n\nu$.
Then 
\[
\left( \SM(\nu,n), \langle \partial \rangle, K_{\lambda/\mu,n^m}(q)  \right)
\]
is a CSP-triple.
\end{theorem}
\begin{proof}
We first note that $\partial$ acting on $T\in \SM(\nu,n)$ simply corresponds to 
cyclic rotation of the columns one step left in $M(T)$. 
Hence, elements in $\SM(\nu,n)$ fixed under $\partial^d$ are in bijection with 
matrices in $\matrixSet(n\nu,n^{m})$ which are 
fixed under cyclic rotation of columns $j$ steps.
A result by B.~Rhoades \cite[Thm. 1.3]{Rhoades2010b}, with a suitable specialization, 
implies that
\[
\left( M(n\nu,n^{m}), \langle \text{cyclic rotation of columns} \rangle,
\sum_{\lambda \vdash mn} K_{\lambda,n^m}(q) K_{\lambda,n\nu}(1) \right)
\]
is a triple which exhibits the cyclic sieving phenomenon.
Together with \cref{cor:skewCharge}, the theorem follows.
\end{proof}

\subsection{Disjoint union of rectangles}

We shall now extend \cref{thm:CSPonStretchedRibbons} to a much
larger class of skew shapes and contents, where each
disjoint shape is now a rectangle (and not just a row), and arbitrary contents.

In \cite{Rhoades2010}, B.~Rhoades proved that promotion 
acting on rectangular semistandard Young tableaux of shape $a^b$, and entries bounded by $m$,
gives an instance of cyclic sieving together with evaluating $q^\ast\schurS_{a^b}(1,q,\dotsc,q^{m-1})$.
This was later refined in \cite{FontaineKamnitzer2013}, where the content of 
the semistandard tableaux was held fixed, and the CSP-polynomial is essentially given by 
a Kostka--Foulkes polynomial.

In this subsection, we extend their result slightly, to skew shapes consisting of 
a disjoint union of rectangles. First some additional notation.
We let $\mdefin{\mu^k}$ denote the partition (composition) obtained from $\mu$,
where each part has been repeated $k$ times,
so that if $\mu = 1^{a_1} 2^{a_2} \dotsb \ell^{a_\ell}$ then
$\mdefin{\mu^k} \coloneqq 1^{k\cdot a_1} 2^{k\cdot a_2} \dotsb \ell^{k\cdot a_\ell}$.
For compositions, we concatenate the parts, e.g. $125^3 = 125 125 125$.

A skew shape with $k$ boxes is a \defin{$k$-ribbon} if it is connected and does not 
contain a $2\times 2$-arrangement of boxes. 
The \defin{head} of a ribbon is the upper right-most box.
A collection of $k$-ribbons form a \defin{horizontal strip} if
their union is a skew shape and their heads lie in different columns.
A \defin{semistandard $k$-ribbon tableau} of shape $\lambda/\mu$ and
content $\nu$ is a sequence of skew shifted diagrams
$\lambda_1/\mu \subset \lambda_2/\mu \subset \cdots \subset \lambda_n/\mu=\lambda/\mu$ where $\lambda_1/\mu$
is a horizontal ribbon strip containing $\nu_1$ $k$-ribbons
and each $\lambda_{i+1}/\lambda_{i+1}$ is a horizontal
$k$-ribbon strip containing $\nu_i$ $k$-ribbons.

We let $\mdefin{K^{(k)}_{\lambda/\mu,\nu}}$ be the number of
\defin{semistandard $k$-ribbon tableaux} of shape $\lambda/\mu$ and content $\nu$.
Note that for $k=1$, we recover the usual Kostka coefficient $K_{\lambda/\mu,\nu}$.
Finally, if $\lambda/\mu$ is a ribbon with $k$ boxes, 
we let $\mdefin{\varepsilon_k(\lambda/\mu)} \coloneqq (-1)^{h-1}$
where $h$ is the number of rows covered by $\lambda/\mu$.
For arbitrary skew shapes $\lambda/\mu$, we let 
\[
\mdefin{\varepsilon_k(\lambda/\mu)} \coloneqq \prod_j 
\varepsilon_k(\lambda^{(j)}/\mu^{(j)})
\]
where $\lambda^{(1)}/\mu^{(1)}$, 
$\lambda^{(2)}/\mu^{(2)},\dotsc$ is any partitioning of $\lambda/\mu$
with $k$-ribbons (the sign can be shown to be independent of the particular choice of partitioning).
We let $\varepsilon_k(\lambda/\mu)$ be $0$ if no such partitioning exists.
\begin{example}\label{ex:kRibbonTab}
The $3$-ribbon tableaux counted by $K^{(3)}_{4422,1111}$
are the following six fillings,
\[ 
\includegraphics[scale=0.8]{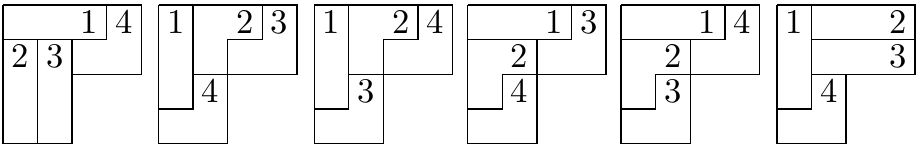} 
\]
and the $3$-ribbon tableaux counted by $K^{(3)}_{4422,211}$
are given by
\[ 
\includegraphics[scale=0.8]{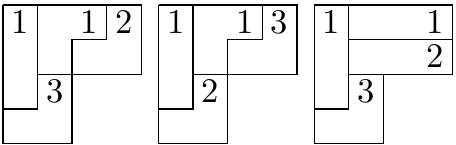}.
\]
It is straightforward to verify that $\varepsilon_3(4422)=-1$, for example,
as the last semistandard $3$-ribbon tableau gives a partitioning with four $3$-ribbons, 
with the signs $(-1)^{3-1}$, $(-1)^{1-1}$, $(-1)^{1-1}$, and $(-1)^{2-1}$, respectively.
\end{example}
It can be shown (see e.g., \cite{AlexanderssonPfannererRubeyUhlin2020x}) that $\varepsilon_k(\lambda/\mu) $ can be given in terms of a skew character, but we do not need that here.
%\[
 %\varepsilon_k(\lambda/\mu) = \sign\left( \chi^{\lambda/\mu}(k^d) \right)
%\]
%where $\chi^{\lambda/\mu}$ is a skew character, and $\lambda/\mu$ is a skew shape 
%of size $kd$.

\begin{theorem}[{See \cite[p.29]{DesarmenienLeclercThibon1994}}]\label{thm:skewKostkaAtRootsOfUnity}
Let $\lambda/\mu$ be a skew shape and $\nu$ a weak composition.
Let $\xi$ be a primitive $j^\thsup$ root of unity.
Then
 \[
  K^{(j)}_{\lambda/\mu,\nu} =  (-1)^{|\nu|(j-1)}\varepsilon_j(\lambda/\mu) K_{\lambda/\mu,\nu^j}(\xi).
 \]
\end{theorem}
For example, $K_{4422,2^31^31^3}(q)$ is given by
\begin{align*}
 &q^{25}+q^{24}+4 q^{23}+5 q^{22}+10 q^{21}+13 q^{20}+21 q^{19}+ \\
 &24 q^{18}+33 q^{17}+34 q^{16}+39 q^{15}+36 q^{14}+36 q^{13}+ \\
 &27 q^{12}+23 q^{11}+14 q^{10}+9 q^9+4 q^8+2 q^7
\end{align*}
and it evaluates to $-3$ at $q=e^{2\pi i/3}$. This is in agreement with 
the fact that $K^{(3)}_{4422,211} = 3$, as we saw in \cref{ex:kRibbonTab}. Note that basic properties of Kostka-Foulkes polynomials allows to reorder the composition giving the content, in this case 
$K_{4422,2^31^31^3}(q)=K_{4422,211211211}(q)$ .

We now recall the main result in \cite{FontaineKamnitzer2013}.
\begin{theorem}[Cyclic sieving on rectangular SSYT, fixed content]\label{thm:oneRectangleSSYTCSP}
 Let $(\gamma_1,\dotsc,\gamma_m)$ be a sequence of non-negative integers 
 with sum $ab$. Suppose $\rot_m^d(\gamma) = \gamma$ for some $d \mid m$.
 Then 
 \[
  \left( \SSYT(a^b,\gamma), \langle \partial^d \rangle, q^{\frac{1}{2}(a^2b - (\gamma_1^2+\dotsb+\gamma_m^2))}
  K_{a^b,\gamma}(q) \right)
 \]
is a CSP-triple. Note that $\langle \partial^d \rangle$ generates a cyclic group of size $m/d$.
\end{theorem}
Note that $\rot_m^{d}(\gamma) = \gamma$ implies that $\gamma$
is the concatenation of $m/d$ copies of some composition with $d$ parts.
Let  $\mu \coloneqq  (\gamma_1,\dotsc,\gamma_{jd})$ and with a light abuse of notation we will also write $\mu= \gamma^{jd/m}$.
For example, $\gamma = 12121212$, $d=2$, $j=2$ gives $\mu=1212$.

Combining this theorem with \cref{thm:skewKostkaAtRootsOfUnity}, 
we have that for fixed $a^b$, $d \geq 1$ and $\gamma = (\gamma_1,\dotsc,\gamma_m)$ and any $j \mid \frac{m}{d}$,
that
\begin{equation}\label{eq:fixedPts}
 |\{ T \in \SSYT(a^b,\gamma) \} :   \partial^{jd}(T) = T \}| = 
 \begin{cases}
 K^{(m/(jd))}_{a^b,\gamma^{jd/m}}  &\text{if } \rot_m^{d}(\gamma) = \gamma \\
 0 &\text{otherwise}.
 \end{cases}
 \end{equation}

Let $\ell= \frac{m}{jd}$ and $\xi$ be a primitive $\ell^\thsup$ root of unity.
From \eqref{eq:fixedPts} and
\cref{thm:oneRectangleSSYTCSP} (with the assumption that 
$\rot_m^{d}(\gamma) = \gamma$) we have due to the CSP that
\begin{equation}
K^{(\ell)}_{a^b,\gamma^{1/\ell}} =
  \xi^{\frac{1}{2}(a^2b - (\gamma_1^2+\dotsb+\gamma_m^2))}
  K_{a^b,\gamma}(\xi).
\end{equation}
Now, by \cref{thm:skewKostkaAtRootsOfUnity}, 
$K^{(\ell)}_{a^b,\gamma^{1/\ell}} = \varepsilon_\ell(a^{b}) (-1)^{|\gamma^{1/\ell}|(\ell-1)} K_{a^b,\gamma}(\xi)$,
so (unless both sides vanish) 
\begin{align}\label{eq:signRelation}
 (-1)^{\frac{ab}{\ell}(\ell-1)}
 \varepsilon_\ell(a^b) 
 &=  \xi^{\frac{1}{2}(a^2b - (\gamma^2_1 + \dotsb + \gamma_m^2))}.
\end{align}
Note that we must have $\ell \mid ab$ in order for $K^{(\ell)}_{a^b,\gamma^{1/\ell}}$
to be non-zero, so $(-1)^{\frac{ab}{\ell}(\ell-1)}$ is either $-1$ or $1$.
Moreover, $\gamma^2_1 + \dotsb + \gamma_m^2=\ell(\gamma^2_1+ \dotsb + \gamma_{jd}^2)$,
so there is a factor of $\ell/2$ in the exponent of $\xi$.

Intuitively, we can then think that
%\begin{align*}
$\xi^{\gamma^2_1 + \dotsb + \gamma_m^2} 
 = (-1)^{\gamma^2_1 + \dotsb + \gamma_{jd}^2} 
 = (-1)^{\gamma_1 + \dotsb + \gamma_{jd}} 
 = (-1)^{\frac{ab}{\ell}}$.
%\end{align*}
Hence, the appearance of $\gamma$ in \eqref{eq:signRelation}
is only used to define a ``nice'' exponent, where the dependence on $\ell$
is encapsulated in $\xi$.
\medskip 

\begin{corollary}\label{cor:signCorrection}
Given positive integers $a$, $b$ and $M$, where $M\mid ab$, there is an integer $E = E(a,b,M)>0$, such that
for all $\ell \mid M$ with $\varepsilon_\ell(a^b) \neq 0$, we have
\begin{align}\label{eq:signRelation2}
 (-1)^{\frac{ab}{\ell}(\ell-1)}
 \varepsilon_\ell(a^b) 
 &=  \xi^{E},
\end{align}
where $\xi$ is a primitive $\ell^\thsup$ root of unity.
\end{corollary}
\begin{proof}
Use \eqref{eq:signRelation}, we can choose $m$ and $d$ such that $M = \frac{m}{d}$ and a $\gamma$, satisfying $\rot_m^d(\gamma)=\gamma$, $\sum \gamma_i=ab$.
\end{proof}

We shall now use \cref{cor:signCorrection} to study shapes which are 
disjoint unions of rectangles.
For positive integer vectors $\avec=(a_1,\dots,a_r), \bvec=(b_1,\dots,b_r)$, let 
\[
\mdefin{\avec^\bvec} \coloneqq (a_1^{b_1}) \oplus \dotsb \oplus (a_r^{b_r})
\]
denote a skew shape consisting of a disjoint union of $r$ rectangles,
where rectangle $k$, $1\leq k \leq r$, has shape $a_k^{b_k}$.

\begin{lemma}\label{lem:sgnSeveralRects}
 Let $\avec^\bvec$ be a disjoint union of rectangles, 
 and $\gamma$ be a composition of length $m$ and size $a_1b_1+\dotsb + a_r b_r$,
 such that $\rot_m^d(\gamma) = \gamma$ for some $d \mid m$.
 Then there is an $E \in \setN$ (depending on $\avec^\bvec$, $m$ and $d$), 
 such that for all $j \mid \frac{m}{d}$ where $K_{\avec^\bvec,\gamma}(\xi) \neq 0$,
 \begin{equation}
  \xi^{E} =  \sign K_{\avec^\bvec,\gamma}(\xi),
\end{equation}
where $\xi$ is a primitive $j^\thsup$ root of unity.
\end{lemma}
\begin{proof}
First of all, \cref{thm:skewKostkaAtRootsOfUnity} implies that unless $j$ 
divides $a_k b_k$ for all $k \in \{1,\dotsc,r\}$, we have $K_{\avec^\bvec,\gamma}(\xi) = 0$
where $\xi$ is a primitive $j^\thsup$ root of unity.
Hence, it suffices to consider the case when 
\begin{equation}\label{eq:dividesCondition}
 \frac{m}{d} \quad \text{ divides } \quad \gcd( a_1 b_1, a_2 b_2,\dotsc,a_r b_r).
\end{equation}
By \cref{thm:skewKostkaAtRootsOfUnity} and some rewriting, we have that 
\begin{align*}
 \sign K_{\avec^\bvec,\gamma}(\xi)
 &=  (-1)^{\frac{1}{j}|\gamma| (j-1) } \varepsilon_{j}(\avec^\bvec) \\
 &= (-1)^{\frac{1}{j}|a_1b_1 + \dotsb + a_r b_r| (j-1) } \prod_{k=1}^r 
 \varepsilon_{j}(a_k^{b_k}) \\
 &= \prod_{k=1}^r 
 (-1)^{\frac{a_kb_k}{j}(j-1) }
 \varepsilon_{j}(a_k^{b_k}).
 \end{align*}
 We can now use \cref{cor:signCorrection}---since we assume \eqref{eq:dividesCondition}---on
 each of the factors in the right hand side, and deduce that
 $
  \sign K_{\avec^\bvec,\gamma}(\xi) = \xi^{E_1 + \dotsb + E_r} = \xi^{E}
 $
for some fixed $E$ which does not depend on $j$, but only on $\avec^\bvec$, $m$ and $d$.
\end{proof}

We can now prove the main result of this subsection.
\begin{theorem}[Cyclic sieving on disjoint rectangles]\label{thm:rectanglesCSP}
Suppose $\gamma = (\gamma_1,\dotsc,\gamma_m)$
is an integer vector with total sum $a_1b_1+\dotsb + a_r b_r$,
and such that $\rot_m^d(\gamma) = \gamma$.
Then there exists some $E = E(\avec^\bvec,\gamma) \in \setN$ such that
\[
\left( 
\SSYT(\avec^\bvec, \gamma),
\langle \partial^d \rangle,
q^{E} K_{\avec^\bvec, \gamma}(q)
\right)
\]
is a CSP-triple.
\end{theorem}

Before the proof, let us briefly discuss some details.
Promotion rotates the content, $\gamma$,
so in order for $\partial^d$ to fix an element, it is trivially necessary 
that $\gamma$ is fixed under $\rot_m^d$.
However, since promotion act independently on each rectangle,
we must have that the content on the $k^\thsup$ rectangle, $\nu^{(k)}$,
also has this rotational symmetry in order for a tableau to be a fixed-point,
for every $k=1,\dotsc,r$.

\begin{proof}
Let $\xi$ be a primitive $j^\thsup$ root of unity, 
where $j \mid \frac{m}{d}$. 
It is given that $\gamma$ is the concatenation of
$m/d$ copies of some smaller composition, so we can 
set $\mu \coloneqq \gamma^{jd/m} = (\gamma_1,\dotsc,\gamma_{jd})$.

By using \cref{lem:sgnSeveralRects}, we can conclude that there is an $E>0$,
depending only on $\avec^\bvec$ and $\gamma$, such that
$q^E K_{\avec^\bvec, \gamma}(q)$ is a non-negative integer as $q = \xi$.
Hence, we have that
$\xi^{E} K_{\avec^\bvec, \gamma}(\xi)$
is equal to $K^{(m/(jd))}_{\avec^\bvec,\mu}$,
and it remains to show that
\begin{equation}\label{eq:rectanglesEquality}
 K^{(m/(jd))}_{\avec^\bvec,\mu} = |\{ T \in \SSYT(\avec^\bvec,\gamma) \} :   \partial^{jd}(T) = T \}|. 
\end{equation}
Since $\partial$ act on each rectangle independently, we have that
the right hand side of \eqref{eq:rectanglesEquality} is given by
\[
\sum_{\nu^{(1)} + \dotsb + \nu^{(r)} = \gamma}
\prod_{k=1}^r
|\{ T \in \SSYT(a_k^{b_k},\nu^{(k)}) \} : \partial^{jd}(T) = T \}|,
\]
as we need to distribute the entries in $\gamma$ among the different rectangles.
However, the product is $0$ unless each composition $\nu^{(i)}$ has the
rotational symmetry $\rot^{jd}_m(\nu^{(i)}) = \nu^{(i)}$,
so we have
\begin{equation}\label{eq:rectanglesEqualityRHS}
\sum_{\rho^{(1)} + \dotsb + \rho^{(r)} = \mu}
\prod_{k=1}^r
|\{ T \in \SSYT(a_k^{b_k},(\rho^{(k)})^{m/(jd)} ) \} : \partial^{jd}(T) = T \}|.
\end{equation}
Now, the left hand side of \eqref{eq:rectanglesEquality} is given by
\begin{equation}\label{eq:rectanglesEqualityLHS}
\sum_{\rho^{(1)} + \dotsb + \rho^{(r)} = \mu}
\prod_{k=1}^r
K^{(m/(jd))}_{a_k^{b_k},\rho^{(k)}}
\end{equation}
since any skew semistandard ribbon tableau of shape $\avec^\bvec$
and content $\mu$ is formed from $r$ semistandard ribbon tableaux of rectangular shape,
where the total content is $\mu$.
We can now see that \eqref{eq:rectanglesEqualityRHS} and \eqref{eq:rectanglesEqualityLHS}
agree since by \cref{thm:skewKostkaAtRootsOfUnity}, for every $k \in \{1,\dotsc,r\}$,
\[
 |\{ T \in \SSYT(a_k^{b_k},(\rho^{(k)})^{m/(jd)} ) \} : \partial^{jd}(T) = T \}|
 = K^{(m/(jd))}_{a_k^{b_k},\rho^{(k)}}.
\]
Hence, we have proved the CSP.
\end{proof}

In the above proof, we were able to adjust the sign of $K_{\avec^\bvec, \gamma}(q)$,
by multiplying with an appropriate power of $q$, so that 
the result is a CSP-polynomial.
The following example illustrates that adjusting the
sign of a potential CSP-polynomial is not always possible.
\begin{example}
Let 
\[
\begin{cases}
 f(q) &= 6 + 2 q + 3 q^2 + 2 q^3 + 3 q^4 + 2 q^5 \\
 g(q) &=  4 + 3 q + 4 q^2 + 4 q^4 + 3 q^5.
\end{cases}
\]
If we let $\xi$ be a primitive $6^\thsup$ root of unity, then 
\begin{align*}
\left( f(\xi^1),f(\xi^2),f(\xi^3),f(\xi^6) \right) &= (3,3,6,18) \\
\left( g(\xi^1),g(\xi^2),g(\xi^3),g(\xi^6) \right) &= (3,-3,6,18).
\end{align*}
One can show that there is some $X$ of cardinality $18$, 
such that $\langle X, \setZ/6\setZ, f(q) \rangle$ is a CSP-triple.
However, there is no $E \in \setZ$ 
such that $q^E g(q)$ is a non-negative integer at every $6^\thsup$ root of unity.
\end{example}

\section{Bi-cyclic sieving on ribbon SYT}\label{sec:biCSP}

A natural generalization of CSP is when the product of
two cyclic groups act simultaneously on the set. 
It is called bicyclic sieving and was first considered
in \cite{BarceloReinerStanton2008} and can be defined as follows. 
Assume we have two cyclic groups $C_1$, $C_2$ with generators $c_1$, $c_2$ of order $k_1$, $k_2$ respectively, 
acting on a finite set $X$. Let $f(q,t)$ be a bivarite polynomial and $\zeta_1$, $\zeta_2$ be primitive 
$k_1$ and $k_2$-roots of unity respectively.
Then we say that $(X,C_1\times C_2, f(q,t))$ 
exhibits the \defin{bicyclic sieving phenomenon}, biCSP for short,
if for any $i,j\in \setZ $ we have 
\[
f(\zeta_1^i,\zeta_2^j)=|\{x\in X: c_1^ic_2^j \cdot x=x \}|. 
\]
That is, we have cyclic sieving for both cyclic groups, not only separately but also jointly.
We prove biCSP for two families of ribbon SYT in this subsection.

Let $\mdefin{\SYT_R(\alpha_1,\dotsc,\alpha_\ell)}$ denote the set 
of ribbon standard Young tableaux with $\alpha_i$ boxes in row $i$.
\begin{remark}
The cardinality of $\SYT_R(m-b,b)$ is $\binom{m}{b}-1$, 
as there are $\binom{m}{b}$ ways to choose the second row, 
and all choices but $1,2,\dotsc,b$ give a valid standard filling.
\end{remark}

\begin{theorem}
The action  $\partial$ act on $\SYT_R(m-b,b)$ has a unique orbit of size $m-1$,
and all other orbits have sizes dividing $m$.
In particular,
\[
\left( \SYT_R(m-b,b) , \langle \partial^{m}\rangle,
\binom{m}{b} - m + [m-1]_q
\right)
\] and
\[
\left( \SYT_R(m-b,b) , \langle \partial^{m-1}\rangle,
\qbinom{m}{b}_q  - [m]_q + n-1
\right)
\]
are CSP-triples.
\end{theorem}
\begin{proof}

Let us consider the action of promotion on the ribbon.
Let $w=w_1\dots w_m$ be the reading word of the tableau. Either $w_1=1$ or $w_{b+1}=1$. 
If $w_{1}=1$ or if $w_{b+1}=1$ and $w_{b}<w_{b+2}$, the promotion acts separately on each row.
\[ 
\ytableausetup{boxsize=2em}
\begin{ytableau}
\none & \none & \none &\scriptstyle   w_{b+1} & \scriptstyle w_{b+2} & \cdots  &  w_{m} \\
w_{1}& w_{2} & \cdots  & w_b\\
\end{ytableau}
\]
Only if $w_{b+1}=1$ and $w_{b+2}<w_{b}$ an exchange happens
between rows, and that is only possible when the first
row is $1, b+2, b+3, \dotsc, m$. 
Let us call this particular tableau $T^*$. 
In the figure below we will use the notation $a\coloneqq m-b$ for convenience.
The orbit of $T^*$ is of size $m-1$:
\[
\ytableausetup{boxsize=2em}
\begin{ytableau}
\none & \none & \none &  \none &   1 &  b\+2 &  b\+3 & \cdots  &  m \\
2& 3& \cdots  & \cdots & b \+ 1 \\
\end{ytableau}
\]

\[
\xrightarrow{\partial}\quad
\begin{ytableau}
\none & \none & \none & \none & b & b\+1 & \cdots   & \cdots &  m\minus 1 \\
1& 2& \cdots & b \minus 1 & m \\
\end{ytableau}\]

\[
\xrightarrow{\partial^{b-1}}\quad
\begin{ytableau}
\none & \none & \none & \none &   1 & 2 & \cdots  & \cdots &  m\minus b \\
\scriptstyle a\+ 1& \scriptstyle a\+ 2& \cdots &  \cdots & m \\
\end{ytableau}
\]

\[
\xrightarrow{\partial^{a-1}}\quad
\begin{ytableau}
\none & \none & \none &  \none &   1 &  b\+2 &  b\+3 & \cdots  &  m \\
2& 3& \cdots  & \cdots & b \+ 1 \\
\end{ytableau}
\]

In any other orbit, there is no exchange between the rows,
so promotion affects them independently. 
As each row is a rectangular tableaux and our alphabet has $m$ elements,
the order of promotion divides $m$. 

The first CSP follows as $\langle \partial^{m}\rangle$ has order $m-1$,
and we have one orbit of size $m-1$, plus $\binom{m-1}{b}-m$
elements that are fixed by everything.

For the second CSP, note that $\partial^{m-1}$ fixes the $(m-1)$-cycle described above.
For the rest of the tableaux, the action of $\partial^{m-1}$ on the first row
matches the rotation action on $b$ element subsets of $m$,
except that we are missing the elements belonging to the $(m-1)$-cycle
and the subset $\{m-b+1,m-b+2,\dotsc,m\}$ which violates the tableau rules.

Note that under the rotation action these element actually form
an $m$-cycle whose fixed points can be calculated by plugging
in the appropriate root of unity to $[m]_q$.
Subtracting this from the polynomial $\qbinom{m}{b}_q$ given
by the CSP of rotation and adding back
the $m-1$ fixed points gives us the desired result. 
\end{proof}

These two can be combined to give a CSP for the action of $\partial$, 
but we will instead give a bicyclic version as the polynomial is nicer.
\begin{corollary}\label{cor:firstBiCSP}
Promotion $\partial$ acting on $\SYT_R(m-b,b)$ has order $m(m-1)$. 
Let $\psi_{m,b}(q,t) \coloneqq  [m-1]_q + \binom{m}{b}_t - [m]_t$,
let $\xi$ be a primitive $(m-1)^\thsup$ root of unity, and $\zeta$
be a primitive $m^\thsup$ root of unity. Then for all $r,s \in \setZ$,
\[
\psi_{m,b}(\xi^r, \zeta^s) = 
|\{ T \in \SYT_R(m-b,b) : \partial^{r m + s (m-1)}(T) = T  \}|.
\]
In other words, 
\[
\left( \SYT_R(m-b,b) , \langle \partial^{m}\rangle \times \langle \partial^{m-1}\rangle, 
\binom{m}{b}_t - [m]_t + [m-1]_q
\right)
\]
exhibits the bi-cyclic sieving phenomenon. 
\end{corollary}

If $b=1$ or $m-1$, the total number of tableaux is $m-1$, so the $(m-1)$-orbit
is the only one. If $b=m-b=2$, we have $5$ tableaux in total,
which are divided into a $3$-orbit and a $2$-orbit,
so the promotion has order $6$.
Next, we show that apart from these trivial cases,
promotion on two rows has order $m(m-1)$.
\begin{proposition} 
If $b,m-b>1$ and we do not have $b=m-b=2$,
$\partial$ on $\SYT_R(m-b,b)$ has order $m(m-1)$.
\end{proposition}
\begin{proof} 
We will show that in these cases, there is always an orbit of size $m$.
Consider the tableaux with bottom row $1,2,\dotsc,b-1, b+1$.
This tableau does not come up in the $(m-1)$-cycle, so promotion
is applied independently to the two rows.
As promotion has order $m$ on the bottom row, it has order $m$ on the tableau.
\end{proof}

Note that the same arguments apply to the two column case by symmetry.
Next, we consider the three row ribbon case where
the first and last rows consist of one box only.

\begin{lemma} 
We have that
$|\SYT_R(1,m-2,1)|=(m-1)(m-2)-1$.
\end{lemma}
\begin{proof}
There are $m-1$ ways to pick the entry to go into the first row,
as it can not be $m$, and $m-2$ ways to pick the entry in
the last row so that it won't be the smallest entry of the rest.
Of these $(m-1)(m-2)$ choices, only one does not give a
tableau---choosing $m-1$ for top row and $m$ for the bottom row.
\end{proof}

\begin{theorem}\label{thm:secondBiCSP}
For $m>3$, the $\partial$ acting on $\SYT_R(1,m-2,1)$ has order $(m-1)(m-2)$,
with one $(m-2)$-cycle, and $m-3$ $(m-1)$-cycles.
As a consequence, if we let 
$\mdefin{\psi_{m,b}(q,t)} \coloneqq [m-2]_t + (m-3)[m-1]_q $,
$\xi$ be a primitive $(m-1)^\thsup$ root of unity, and $\zeta$
be a primitive $(m-2)^\thsup$ root of unity, we have $r,s \in \setZ$,
\[
\psi_{m,b}(\xi^r, \zeta^s) = 
|\{ T \in \SYT_R(1,m-2,1) : \partial^{r (m-2) + s (m-1)}(T) = T  \}|.
\]
In other words, 
\[
\left( \SYT_R(1,m-2,1) , \langle \partial^{m-1}\rangle \times \langle \partial^{m-2}\rangle, 
[m-2]_t + (m-3)[m-1]_q
\right)
\]
exhibits the bi-cyclic sieving phenomenon. 
\end{theorem}
\begin{proof} 
There are $m-2$ tableaux where the top row is $1$. 
If the bottom row is $m$, we get the following $m-2$ cycle:
\[ 
\ytableausetup{boxsize=2em}{\color{white}{\xrightarrow{\partial^{m-k}}\quad }}
\begin{ytableau}
\none & \none & \none & \none & 1\\
2 &  3 &\cdots & m\minus 2 &  m\minus 1 \\
m 
\end{ytableau}\quad
\xrightarrow{\partial}\quad
\begin{ytableau}
\none & \none & \none & \none & m\minus 2\\
1 & 2 & \cdots & m\minus 3 &  m\\
m\minus 1
\end{ytableau}\]

\[
\xrightarrow{\partial^{m-4}}\quad \begin{ytableau}
\none & \none & \none & \none & 2\\
1 & 4 & 5 &\cdots &  m\\
3
\end{ytableau} \quad
\xrightarrow{\partial}\quad \begin{ytableau}
\none & \none & \none & \none & 1\\
2 & 3 & 4 &\cdots &  m \minus 1\\
m
\end{ytableau}
\]

For the other $m-3$ cases where bottom row contains $k<m$
we get the following $m-1$ cycles (we denote $m-k$ by $a$ for visual clarity): 
\[ \ytableausetup{boxsize=2em}{\color{white}{\xrightarrow{\partial^{m-k}}\quad }}
\begin{ytableau}
	\none & \none  & \none & \none & \none & \none & 1\\
2 &  3 &  \cdots & k\minus 1 &k\+1& \cdots &  m \\
k	 
\end{ytableau}\quad
\xrightarrow{\partial}\quad
\begin{ytableau}
\none & \none  & \none & \none & \none & \none & m\minus 1\\
1 &  2 &  \cdots & k\minus2 &k& \cdots &  m \\
k\minus 1
\end{ytableau}\]

\[
\xrightarrow{\partial^{k-3}}\quad \begin{ytableau}
\none & \none  & \none & \none & \none & \none & a\+2\\
1 &  3 & \cdots & a\+1 &a\+3& \cdots &  m \\
2
\end{ytableau} \quad
\xrightarrow{\partial}\quad \begin{ytableau}
\none & \none  & \none & \none & \none & \none & a\+1\\
1 &  2 & \cdots & a&a\+2& \cdots &  m \minus 1 \\
m
\end{ytableau}
\]

\[
\xrightarrow{\partial^{m-k}}\quad \begin{ytableau}
\none & \none  & \none & \none & \none & \none & 1\\
2 &  3 &  \cdots & k\minus 1 &k\+1& \cdots &  m \\
k
\end{ytableau}.
\]
\end{proof}

\section{Discussion}

We have covered promotion on a few new classes of shapes,
and in particular, paid some well-deserved attention to skew shapes.
In many cases, the Kostka--Foulkes polynomials play a central role,
except in the last section where we needed something different.

It is natural to examine other families of ribbon shapes.
However, the order of promotion on the set $\SYT_R(k,k,k)$
for $k=1,\dotsc,4$ is $1$, $60$, $814773960$, and $82008289440$, respectively,
which is discouraging.
Similarly, the order of promotion on the set $\SYT_R(1,1,k,1)$
for $k\geq 1$ seem (we have verified this for $k\leq 15$)
to be given by the generating function
\[
 x\frac{1+16x-19x^2 +10x^3 -2 x^4}{(1-x)^4}=
 x+ 20x^2 + 55x^3 +114x^4 + 203x^5 + \dotsb.
\]

The main interesting open case is the staircase.
Let $\mdefin{sc_m} \coloneqq (m,m-1,\dotsc,2,1)$.
It has been shown (see \cite{Haiman1992}) 
that $\partial$ acting on $\SYT(sc_m)$
has order $m(m-1)$.

The staircase together with promotion was 
first considered in the context of CSP in \cite{PonWang2011},
where the authors give an injection from 
$\SYT(sc_m)$ to $\SYT(m^{m+1})$, which commutes with promotion.
However, there is still no natural polynomial 
which gives an instance of the cyclic sieving phenomenon.
\begin{problem}
 Find some $p_m(q) \in \setN[q]$ such that
 \[
  \left( \SYT(sc_m), \langle \partial \rangle,  p_m(q) \right)
 \]
is a CSP-triple.
\end{problem}
Promotion on the \emph{shifted staircase} admits a CSP;
this is an example of a minuscule poset which 
behave nice with respect to promotion and cyclic sieving,
see \cite{RushShi2012}. For more background and 
discussion, see \cite{Hopkins2020}.

In \cite[Conj. 5.2]{Hopkins2019}, a related problem
regarding cyclic sieving on triangular
\emph{plane partitions} is stated, with an explicit polynomial.

\medskip 
Finally, a different approach would be to keep the Kostka--Foulkes polynomials,
and replace promotion with a different group action.
In \cite{AlexanderssonPfannererRubeyUhlin2020x}, it is shown
that the set of standard Young tableaux of shape 
$m\lambda \coloneqq (m\lambda_1,m\lambda_2,\dotsc,m\lambda_\ell)$
admit a cyclic sieving phenomenon with $K_{m\lambda,1^{|m\lambda|}}(q)$
as CSP-polynomial and \emph{some} cyclic group of order $m$.

\medskip 

As a more concrete open problem, the statement of 
\cref{thm:rectanglesCSP} with $h(\avec^\bvec,\gamma)$ is a bit 
unintuitive, as there is a choice of $\delta^{(j)}$:s involved.
\begin{problem}[] Is there some way to define $h(\avec^\bvec,\gamma)$ without 
this seemingly arbitrary choice?
\end{problem}

\bibliographystyle{alphaurl}
\bibliography{bibliography}

\end{document}